\documentclass[12pt,a4paper]{article}

\usepackage[ngerman,english]{babel}
\usepackage[utf8]{inputenc}
\usepackage[T1]{fontenc}
\usepackage{amsmath}
\usepackage{amssymb}
\usepackage{amsthm}
\usepackage{mathtools}
\usepackage{paralist}
\usepackage{tikz}
\usepackage{graphicx}
\usepackage{hyperref}

\begin{document}

\theoremstyle{plain}
\newtheorem{theorem}{Theorem}[section]
\newtheorem{definition}[theorem]{Definition}
\newtheorem{lemma}[theorem]{Lemma}
\newtheorem{prop}[theorem]{Proposition}
\newtheorem{cor}[theorem]{Corollary}
\newtheorem{conjecture}[theorem]{Conjecture}
\theoremstyle{remark}
\newtheorem{remark}[theorem]{Remark}
\newtheorem{example}[theorem]{Example}

\newcommand{\reg}{\mathrm{reg}}
\newcommand{\charakt}{\mathrm{char}}
\newcommand{\diag}{\mathrm{diag}}
\newcommand{\Tor}{\mathrm{Tor}}
\newcommand{\im}{\mathrm{im}}
\newcommand{\coker}{\mathrm{coker}}
\newcommand{\id}{\mathrm{id}}
\newcommand{\length}{\mathrm{length}}
\newcommand{\LM}{\mathrm{LM}}
\newcommand{\LT}{\mathrm{LT}}
\newcommand{\cone}{\mathrm{cone}}
\newcommand{\ord}{\mathrm{ord}}
\newcommand{\Quot}{\mathrm{Quot}}
\newcommand{\Spec}{\mathrm{Spec}}
\newcommand{\height}{\mathrm{ht}}
\newcommand{\rank}{\mathrm{rank}}
\newcommand{\Ann}{\mathrm{Ann}}
\newcommand{\reynolds}{\mathcal{R}}
\newcommand{\maxId}{\mathfrak{m}}
\newcommand{\maxIdn}{\mathfrak{n}}
\newcommand{\primId}{\mathfrak{p}}

\title{Arithmetic invariants of pseudoreflection groups and regular graded algebras}
\author{David Mundelius \\ \small{Technische Universität München, Zentrum Mathematik - M11} \\ \small{Boltzmannstraße 3, 85748 Garching, Germany} \\ \small{\texttt{david.mundelius@tum.de}}}
\date{August 27, 2020}
\maketitle

\begin{abstract}
The main result of this paper is a generalization of the theorem of Chevalley-Shephard-Todd to the rings of invariants of pseudoreflection groups over Dedekind domains. In the special case of a principal ideal domain in which the group order is invertible it is proved that this ring of invariants is isomorphic to a polynomial ring. An intermediate result is that every finitely generated regular graded algebra over a Dedekind domain is isomorphic to a tensor product of blowup algebras.
\end{abstract}


\section*{Introduction}

The famous theorem of Chevalley-Shephard-Todd states that the ring of invariants of a finite pseudoreflection group $G$ over a ground field in which $|G|$ is invertible is isomorphic to a polynomial ring over this field. The main goal of this paper is to generalize this result from ground fields to ground rings, more specifically to ground rings which are Dedekind domains. The restriction to Dedekind domains is natural, since in the important case of irreducible pseudoreflection groups over the complex numbers, every such group can be defined over the ring of integers of an algebraic number field. In the case where the ground ring is a principal ideal domain we find the especially nice result mentioned in the abstract. This covers many interesting cases; in particular, Feit \cite{feit2} observed that in all the exceptional cases in the classification of irreducible complex pseudoreflection groups by Shephard and Todd \cite{st}, this ring of integers is indeed a principal ideal domain. There is a proof of the theorem of Chevalley-Shephard-Todd due to Smith \cite{smithst}, which is based on the fact that a finitely generated graded algebra over a field is isomorphic to a polynomial ring if and only if its global dimension is finite. A Noetherian ring has finite global dimension if and only if it is regular, and the approach in this paper is based on a similar characterization of all finitely generated regular graded algebras over a Dedekind domain: every such algebra is isomorphic to a tensor product of blowup algebras, see the beginning of Section \ref{SectionBlowup} for the definition; over a principal ideal domain, such a tensor product is always a polynomial ring.

Section \ref{SectionBlowup} contains various basic properties of tensor products of blowup algebras. These are used in Section \ref{SectionAlgebras} to prove that every finitely generated regular graded algebra over a Dedekind domain $R$ is isomorphic to a tensor product of blowup algebras of ideals in $R$ (see Theorem \ref{StructureGradAlgDedekind}), and that such an algebra is isomorphic to a polynomial ring if and only if the degree-$d$-part is a free $R$-module for every $d \in \mathbb{N}$ (see Theorem \ref{StructureGradAlgFree}). In the final section we apply this to rings of invariants of finite pseudoreflection groups over Dedekind domains. The first result is that the ring of invariants of a pseudoreflection group over a Dedekind domain $R$ is isomorphic to a tensor product of blowup algebras if and only if the ring of invariants over the localization $R_\primId$ is a polynomial ring for every maximal ideal $\primId \subset R$. Here as in every other result of Section \ref{SectionInvariants} we can replace "tensor product of blowup algebras" by "polynomial ring" if we assume that $R$ is a principal ideal domain. Then we prove criteria under which rings of invariants over discrete valuation rings are polynomial rings (see Propositions \ref{PolyRingSameDeg} and \ref{PolyRingSameDegEquiv}). By putting these results together, we obtain that the ring of invariants of a pseudoreflection group over a Dedekind domain $R$ is isomorphic to a tensor product of blowup algebras if the group order is invertible in $R$ (see Theorem \ref{InvPseudoGenBlowup}); this is the direct generalization of the theorem of Chevalley-Shephard-Todd. In Theorem \ref{AInvInjGrp} we prove a somewhat more technical generalization of this result to Dedekind domains in which the group order is not invertible. We have no example of an invariant ring of a pseudoreflection groups over a Dedekind domain which is a tensorproduct of blowup algebras but not a polynomial ring and we conjecture that such a ring of invariants does not exist, see Conjecture \ref{ConjPseudoFree}.

\subsection*{Acknowledgements}

The results of this paper also appear in my dissertation \cite{dissMundelius}. I want to thank my Ph.D. advisor Gregor Kemper for proposing the interesting topic and for many fruitful discussions. Further I acknowledge the support from the graduate program TopMath of the Elite Network of Bavaria and the TopMath Graduate Center of TUM Graduate School at Technische Universität München.

\section{Blowup algebras} \label{SectionBlowup}

In this section we develop some basic results on blowup algebras and tensor products of these, which we will need to prove the structure theorem in the next section. We shall always assume that $R$ is a Dedekind domain and $I_1, \ldots, I_n$ are nonzero ideals in $R$. Although the results in this section are quite elementary, I did not find them in the literature. We first recall the definition of a blowup algebra (see \cite{eisenbud}): for a nonzero ideal $I \subseteq R$, the blowup algebra of $I$ in $R$ is the graded algebra $B_I R:= \bigoplus_{i \in \mathbb{N}_0} I^i$. We introduce the following notation for tensor products of blowup algebras:
\[ B_{I_1, \ldots, I_n} R:=B_{I_1} R \otimes_R \ldots \otimes_R B_{I_n} R. \]
We can make each blowup algebra $B_I R$ into a graded ring in several different ways, since we can choose an arbitrary integer $d>0$ and then assign the degree $d$ to all elements of $I$. When we want to make $B_{I_1, \ldots, I_n} R$ into a graded ring, then we can choose one of these gradings for each of the factors $B_{I_i} R$. When we write in the following that some graded ring $S$ is isomorphic to $B_{I_1, \ldots, I_n} R$, then this always means that for one of the gradings on $B_{I_1, \ldots, I_n} R$ defined above there is a graded isomorphism $S \cong B_{I_1, \ldots, I_n} R$; we use the same convention in the special case of polynomial rings $R[x_1, \ldots, x_n]$.

Our first step is to compute certain localizations of the algebras $B_{I_1, \ldots, I_n} R$, which we will need for our later results.

\begin{lemma} \label{LocalizGenBlowup}
Let $P \subseteq R$ be a prime ideal and $U:= R \backslash P$. Then we have $U^{-1} B_{I_1, \ldots, I_n} R \cong R_P[x_1, \ldots, x_n]$.
\end{lemma}

\begin{proof}
Since $U^{-1}(B_{I_1} R \otimes_R \cdots \otimes_R B_{I_n} R) \cong (U^{-1} B_{I_1} R) \otimes_{U^{-1} R} \cdots \otimes_{U^{-1} R} (U^{-1} B_{I_n} R)$ we only have to consider the case $n=1$. In this case we have $U^{-1} B_{I_1} R = \oplus_{i \in \mathbb{N}_0} U^{-1} I_1^i$, and for every $i$, $U^{-1} I_1^i$ is isomorphic to a nonzero ideal in $R_P$. But $R_P$ is a principal ideal domain and thus $U^{-1} I_1^n$ is isomorphic to $R_P$ itself. Together with the fact that $B_{I_1} R$ is generated by its degree-$1$-part as an $R$-algebra, we obtain that indeed $ U^{-1} B_{I_1} R \cong R_P[x]$.
\end{proof}

If $I \neq(0)$ is a principal ideal, then the blowup algebra $B_I R$ is isomorphic to the polynomial ring $R[x]$. As an analogue to the well-known result that $\dim R[x_1, \ldots, x_n] = \dim R + n$, we now compute the Krull dimension of $B_{I_1, \ldots, I_n}$.

\begin{prop} \label{DimGenBlowup}
If $R$ is not a field, the Krull dimension of $B_{I_1, \ldots, I_n} R$ is $n+1$.
\end{prop}

\begin{proof}
Let $P$ be a nonzero prime ideal in $R$. By Lemma \ref{LocalizGenBlowup} and the fact that $R_P$ is a discrete valuation ring we have $\dim B_{I_1, \ldots, I_n} R \geq \dim R_P[x_1, \ldots, x_n] = n+1$. We prove the reverse inequality by induction on $n$: the case $n=0$ is trivial. For $n>0$ we write $S:=B_{I_1, \ldots, I_n} R$ and $T:=B_{I_1, \ldots, I_{n-1}} R$. So by induction we have $\dim T \leq n$, and we want to show $\dim S \leq n+1$. The definitions imply $S=T \otimes_R B_{I_n} R$, and we define a map $\varphi: T \to S, a \mapsto a \otimes 1$. Now let $Q \subset S$ be a prime ideal, and define $P:=\varphi^{-1}(Q)$. It is sufficient to prove that we have $\height(Q) \leq \height(P) +1$. Since $S$ and $T$ are finitely generated $R$-algebras and hence Noetherian, this will follow if we can show that the dimension of the fiber ring $\Quot(T/P) \otimes_T S$ is at most one (see e.g. Kemper \cite[Theorem 7.12]{kemperca}).

We have $\Quot(T/P) \otimes_T S= \Quot(T/P) \otimes_T (T \otimes_R B_{I_n} R)= \Quot(T/P) \otimes_R B_{I_n} R= \bigoplus_{m \in \mathbb{N}_0} (\Quot(T/P) \otimes_R I_n^m)$. As a $K:=\Quot(T/P)$-algebra, this ring is generated by the summand with $m=1$, which is $K \otimes_R I_n$. This is a subset of $K \otimes_R R \cong K$, so its dimension as a $K$-vectorspace is at most one. Hence the $K$-algebra $K \otimes_R B_{I_n} R$ is isomorphic to a quotient of $K[x]$ and hence its Krull dimension is indeed at most one.
\end{proof}

Our next goal is to prove that under our general assumptions $B_{I_1, \ldots, I_n} R$ is always a regular ring. We begin with a well-known lemma, which will be important again in the next section. Since I did not find a proof of this statement in the literature, I give a proof here for convenience.

\begin{lemma} \label{RegHomMaxId}
Let $S=\oplus_{i \geq 0} S_i$ be a graded ring.
\begin{compactenum}[a)]
\item $S$ is regular if and only if the localization $S_\maxId$ is regular for every homogeneous maximal ideal $\maxId \subset S$.
\item The homogeneous maximal ideals in $S$ are precisely the ideals of the form $(\primId, S_+)_S$ where $\primId$ is a maximal ideal in $S_0$.
\end{compactenum}
\end{lemma}

\begin{proof}
\begin{compactenum}[a)]
\item By Bruns and Herzog \cite[Exercise 2.2.24]{bh} we know that $S$ is regular if and only if $S_\primId$ is regular for every homogeneous prime ideal $\primId \subset S$. This together with the definition of regularity already proves the "only if"-part of the statement and for the "if"-part it only remains to show that every homogeneous prime ideal is contained in a homogeneous maximal ideal; then the statement follows from the fact that localizations of regular rings are again regular. So let $\primId \subset S$ be a homogeneous prime ideal. Then $\primId \cap S_0$ is a proper ideal in $S_0$ and thus contained in a maximal ideal $\maxIdn \subset S_0$. Now we define $\maxId:=(\maxIdn, S_+)_S$. This is clearly a homogeneous ideal, and it is maximal because $S/ \maxId \cong S_0 / \maxIdn$. It is also clear that $\primId \subseteq \maxId$.
\item Let $\maxId \subset S$ be a homogeneous maximal ideal. Let $\primId := \maxId \cap S_0$; this defines a prime ideal in $S_0$. Then $\maxIdn:=(\primId, S_+)_S$ is a proper ideal in $S$ with $\maxId \subseteq \maxIdn$, and hence $\maxId = \maxIdn$ since $\maxId$ is a maximal ideal. Now $S/ \maxIdn \cong S_0 / \primId$ and thus $\maxIdn$ is maximal if and only if $\primId$ is a maximal ideal in $S_0$.
\end{compactenum}
\vspace{-1em}
\end{proof}

Now we can prove the following result:

\begin{lemma} \label{GenBlowupRegular}
Under our general assumptions in this section, the ring $B_{I_1, \ldots, I_n} R$ is always a regular ring.
\end{lemma}

\begin{proof}
We write $S:=B_{I_1, \ldots, I_n} R$. From Lemma \ref{RegHomMaxId} we know that we only have to prove that $S_\maxId$ is regular for every ideal $\maxId \subset S$ of the form $\maxId =(\primId, S_+)_S$ where $\primId$ is a maximal ideal in $R$. In this case $S_\maxId$ is a localization of $(R \backslash \primId)^{-1} S$, and by Lemma \ref{LocalizGenBlowup} this last ring is isomorphic to the polynomial ring $R_\primId[x_1, \ldots, x_n]$, which is regular because $R_\primId$ is regular. Now the claim follows from the fact that localizations of regular rings are again regular.
\end{proof}

\section{Regular graded algebras} \label{SectionAlgebras}

In this section, a graded $R$-algebra shall always mean a graded ring $S=\bigoplus_{i \in \mathbb{N}} S_i$ such that $S_0 \cong R$. The main goal of this section is to prove the following theorem:

\begin{theorem} \label{StructureGradAlgDedekind}
Let $R$ be a Dedekind domain and let $S$ be a finitely generated graded $R$-algebra. Then the following two statements are equivalent:
\begin{compactenum}[(i)]
\item $S$ is regular.
\item There exist nonzero ideals $I_1, \ldots, I_n$ in $0$ (not necessarily distinct) such that $S \cong B_{I_1, \ldots, I_n} R$.
\end{compactenum}
\end{theorem}

Since for a nonzero principal ideal $I \subseteq R$ we have $B_I R \cong R[x]$, the following Corollary is an immediate consequence of Theorem \ref{StructureGradAlgDedekind}:

\begin{cor} \label{StructureGradAlgPID}
Let $R$ be a principal ideal domain and let $S$ be a finitely generated graded $R$-algebra. Then the following two statements are equivalent:
\begin{compactenum}[(i)]
\item $S$ is regular.
\item $S$ is isomorphic to a polynomial ring $R[x_1, \ldots, x_n]$.
\end{compactenum}
\end{cor}

We will also prove the following theorem, which generalizes Corollary \ref{StructureGradAlgPID} in another direction than Theorem \ref{StructureGradAlgDedekind}.

\begin{theorem} \label{StructureGradAlgFree}
Let $R$ be a Dedekind domain and let $S=\bigoplus_{i \in \mathbb{N}} S_i$ be a finitely generated graded $R$-algebra. Then the following two statements are equivalent:
\begin{compactenum}[(i)]
\item $S$ is regular and for every $i$, $S_i$ is free as an $R$-module.
\item $S$ is isomorphic to a polynomial ring $R[x_1, \ldots, x_n]$.
\end{compactenum}
\end{theorem}

In the remainder of this section we prove Theorems \ref{StructureGradAlgDedekind} and \ref{StructureGradAlgFree}. These are well known in the case that $R$ is a field, see for example Bruns and Herzog \cite[Exercise 2.2.25]{bh}. We begin the proof of the theorems with the following lemma:

\begin{lemma} \label{RegIntegral}
Let $R$ be a Dedekind domain and let $S=\bigoplus_{i \in \mathbb{N}} S_i$ be a finitely generated graded $R$-algebra. Assume that $S$ is regular. Then $S$ is an integral domain; in particular, all $S_i$ are torsion-free $R$-modules.
\end{lemma}

\begin{proof}
We first prove that $S$ is torsion-free as an $R$-module. For $f \in S \backslash \{ 0 \}$ we define $I_f:=\{ a \in R : a f=0 \}$; so we want to show that $I_f=\{ 0 \}$ for every $f$. It is sufficient to prove this in the case that $f$ is homogeneous. The set $I_f$ is clearly a proper ideal in $R$, so there exists a maximal ideal $\maxIdn_f \subset R$ which contains $I_f$. Next we define $\maxId_f:=(\maxIdn_f \cup S_+)_S$. We have $S/\maxId_f \cong R / \maxIdn_f$ and thus $\maxId_f$ is a maximal ideal in $S$; in particular, we can define the localization $S_{\maxId_f}$ and this is a regular local ring and hence an integral domain. Let $\varepsilon$ be the canonical map $S \to S_{\maxId_f}$. For every $a \in I_f$ we have $\varepsilon(a) \cdot \varepsilon(f) = \varepsilon(0)=0$ and hence $\varepsilon(a)=0$ or $\varepsilon(f)=0$ because $S_{\maxId_f}$ is an integral domain. If $\varepsilon(a)=0$, there exists $c \in U:= S \backslash \maxId_f$ such that $c \cdot a =0$; let $c_0$ denote the degree-$0$-part of $c$. Then we have $c_0 \cdot a =0$, but $c_0 \neq 0$ because $c \notin S_+ \subseteq \maxId_f$. Since $R$ is an integral domain, this implies $a =0$. In the other case, if $\varepsilon(f)=0$, there exists $c \in U$ such that $c \cdot f=0$. Again we define $c_0$ to be the degree-$0$-part of $c$. Then we have $c_0 \cdot f=0$ because $f$ is homogeneous and hence obtain $c_0 \in I_f \subseteq \maxIdn_f \subset \maxId_f$. Together with $c-c_0 \in S_+ \subseteq \maxId_f$, this implies $c \in \maxId_f$ which is a contradiction. So we finally get $I_f = \{ 0 \}$ for every $f \in S \setminus \{ 0 \}$, and thus $S$ is a torsion-free $R$-module.

Now we prove that $S$ is an integral domain. So assume there exist $r,s \in S \backslash \{ 0 \}$ such that $r \cdot s =0$. Since $R$ is an integral domain, $\primId:=S_+$ is a prime ideal in $S$. Let $\eta$ denote the canonical homomorphism $S \to S_\primId$. Since $S$ is regular, $S_\primId$ is a regular local ring and hence an integral domain, so we obtain either $\eta(r)=0$ or $\eta(s)=0$. Without loss of generality, we may assume $\eta(r)=0$. This implies that there is a $t \in S \backslash \primId$ such that $r \cdot t=0$. We write $r=\sum_{i \in \mathbb{N}} r_i$ and $t=\sum_{i \in \mathbb{N}} t_i$ with $r_i, t_i \in S_i$ for every $i$. Let $d \in \mathbb{N}_0$ be the smallest degree such that $r_d \neq 0$. Since $t \notin \primId$, we have $t_0 \neq 0$. Then the degree-$d$-part of $r \cdot t$ is $r_d t_0$, and thus $r \cdot t=0$ implies $r_d \cdot t_0=0$. But we have $t_0 \in R$, and since we have already proved that $S$ is a torsion-free $R$-module, this implies $r_d=0$. This contradicts our choice of $d$, so we have proved that $S$ is an integral domain.
\end{proof}

\begin{remark}
We want to apply the Theorems \ref{StructureGradAlgDedekind} and \ref{StructureGradAlgFree} later to rings of invariants which are clearly integral domains. Nevertheless, Lemma \ref{RegIntegral} is important also in that situation, because we will use induction on $\dim S$ in the proof of these theorems and it is a priori not clear that the integrality hypothesis remains true in the induction step.
\end{remark}

We next prove the following technical lemma which provides the main part of the proof of Theorems \ref{StructureGradAlgDedekind} and \ref{StructureGradAlgFree}.

\begin{lemma} \label{StructureGradAlgLemma}
Let $R$ be a Dedekind domain and let $S=\bigoplus_{i \in \mathbb{N}} S_i$ be a regular graded $R$-algebra such that $S_0 \subsetneq S$. Let $d \in \mathbb{N}_{>0}$ be minimal with $S_d \neq 0$. By the structure theorem for finitely generated torsion-free modules over Dedekind domains (see Jacobson \textup{\cite[Theorem 10.14]{jacba2}}) we can write $S_d=M \oplus I$ where $I$ is isomorphic to some non-zero ideal in $R$ and $M$ is a free $R$-module; set $J \coloneqq (I)_S$. Then the following holds:
\begin{compactenum}[a)]
\item $T \coloneqq S/J$ is again a regular ring.
\item If $S_i$ is a free $R$-module for each $i \in \mathbb{N}_0$, then $T_i$ is also a free $R$-module for each $i$.
\item If $T \cong B_{I_1, \ldots, I_n} R$ for some ideals $I_1, \ldots, I_n$ in $R$, then $S \cong B_{I_1, \ldots, I_n, I} R$.
\end{compactenum}
\end{lemma}

\begin{proof} \
\begin{compactenum}[a)]
\item For the proof of this, it is by Lemma \ref{RegHomMaxId} sufficient to prove that the localization $T_\maxId$ is regular for every homogeneous maximal ideal $\maxId \subset T$. So let $\maxId$ be such an ideal, and let $\maxIdn$ be a homogeneous maximal ideal in $S$ such that $\overline{\maxIdn}=\maxId$, where $\overline{\maxIdn}$ is the image of $\maxIdn$ under the projection map $S \to T$. By Lemma \ref{RegHomMaxId}b) we have $\maxIdn=(\primId, S_+)_S$ for some nonzero prime ideal $\primId \subset R$ and then $T_\maxId \cong S_\maxIdn/J_\maxIdn$. By assumption, $S$ and thus $S_\maxIdn$ is regular; so in order to prove that $S_\maxIdn/J_\maxIdn$ is regular, it is sufficient to prove that there exists a regular system of parameters in $S_\maxIdn$ containing a set of generators of $J_\maxIdn$ (see Bruns and Herzog \cite[Proposition 2.2.4]{bh}). In order to achieve this, we first prove that $J_\maxIdn$ is a principal ideal in $S_\maxIdn$ generated by some $g \in S_d$. For this, we first note that $U_0:=S_0 \backslash \primId$ is a multiplicatively closed subset in $S$ which consists only of homogeneous elements of degree $0$, so the localization $U_0^{-1}S$ is still a graded ring with degree-$0$-part isomorphic to $R_\primId$. The degree-$d$-part of $U_0^{-1} S$ is $I_\primId \oplus M_\primId$, and $I_\primId$ is isomorphic to an ideal in $R_\primId$. But since $R_\primId$ is a discrete valuation ring, we find that this ideal is principal, so $I_\primId$ is generated by one element $g$ and we may choose $g$ in such a way that $g \in I$. Then $g$ also generates $U_0^{-1} J$ as an ideal in $U_0^{-1} S$. But we can also view $S_\maxIdn$ as a localization of $U_0^{-1} S$ and hence $g$ generates $J_\maxIdn$.

So in order to prove that there exists a regular system of parameters in $S_\maxIdn$ containing a set of generators for $J_\maxIdn$, we only have to prove that there exists such a system of parameters containing $g$, or equivalently, that $g \notin (\maxIdn_\maxIdn)^2$. So suppose we would have $g \in (\maxIdn_\maxIdn)^2$. Then there exists an $h \in R \backslash \maxIdn$ such that $g h \in \maxIdn^2$. We write $h=\sum_{i \in \mathbb{N}} h_i$ with $h_i \in S_i$. We have $S_+ \subseteq \maxIdn$ and thus $h_0 \notin \maxIdn$; in particular $h_0 \neq 0$, and because $\maxIdn^2$ is a homogeneous ideal, we have $g h_0 \in \maxIdn^2$. This implies that we can write $gh_0= \sum_j \gamma_j \delta_j$ for certain $\gamma_j, \delta_j \in \maxIdn$. Since $\maxIdn$ is a homogeneous ideal, we may assume that all the $\gamma_j$ and $\delta_j$ are homogeneous. Further $g h_0$ is homogeneous of degree $d$ and by our choice of $d$ we get that in each product $\gamma_j \delta_j$ one factor is of degree $0$ and the other one of degree $d$. So without loss of generality we may assume $\gamma_j \in R \cap \maxIdn= \primId$ and $\delta_j \in S_d \cap \maxIdn =S_d$ for each $j$. We can write each $\delta_j$ as $\delta_j= \lambda_j + \mu_j$ with $\lambda_j \in I$ and $\mu_j \in M$. This implies $g h_0= \sum_j \gamma_j \lambda_j + \sum_j \gamma_j \mu_j$ and using $g \in I$ we obtain $g h_0= \sum_j \gamma_j \lambda_j$. We have chosen $g$ in such a way that it generates $U_0^{-1} I$ as an $U_0^{-1} R$-module. So we have elements $ \eta_j \in U_0^{-1} R$ with $\lambda_j= \eta_j g$, and obtain $g h_0 = \left( \sum_j \gamma_j \eta_j \right)\! g$ and hence $h_0 = \sum_j \gamma_j \eta_j$. Now we can find $\omega \in R \backslash \primId$ such that there exist elements $\theta_j \in R$ with $\eta_j=\frac{\theta_j}{\omega}$ for each $j$. This implies $h_0 \omega = \sum_j \gamma_j \theta_j$, and this is a contradiction because the right hand side is an element of $\primId$ as $\gamma_j \in \primId$ for all $j$, while the left hand side is not. So we proved that $T$ is indeed regular.

\item We already know that $T$ is regular and hence an integral domain by Lemma \ref{RegIntegral}; in particular, all modules $T_i$ are torsion-free and thus projective because $R$ is a Dedekind domain. By the structure theorem for finitely generated torsion-free modules over Dedekind domains (see \cite[Theorem 10.14]{jacba2}), we have $T_i \cong R^l \oplus \tilde{I}$ for some ideal $\tilde{I} \subseteq R$ and some integer $l \geq 0$. Next the projection map $S \to T$ restricts to a surjective homomorphism $S_i \to T_i$ of $R$-modules the kernel of which is $J_i$, which is a free $R$-module: $S_d$ is free, so $I$ is principal and hence $J_i \cong S_{i-d}$. Now since $T_i$ is projective, we get that $S_i \cong T_i \oplus J_i \cong R^l \oplus \tilde{I} \oplus J_i$ and since $S_i$ and $J_i$ are free, we obtain that $\tilde{I}$ is a principal ideal by the structure theorem mentioned above, and thus $T_i$ is indeed free.

\item We have an isomorphism $\alpha: B_{I_1, \ldots, I_n} R \to T$. Then $ \alpha$ maps each $I_j \subseteq B_{I_j} R$ to a submodule of $T_{d_j}$ for some $d_j \in \mathbb{N}$. The canonical projection $\beta: S \to T$ restricts to a surjective homomorphism of $R$-modules $S_{d_j} \to T_{d_j}$. Since $T_{d_j}$ is torsion-free and finitely generated, it is projective, and hence there exists a homomorphism of $R$-modules $\beta'_j: T_{d_j} \to S_{d_j}$ with $\beta \circ \beta'_j = \id$. Define $I'_j :=\beta'(\alpha(I_j))$; then $\beta$ restricts to an isomorphism $I'_j \to \alpha(I_j)$. Since $\alpha$ is an isomorphism, the inverse of the restricted $\beta$ defines a homomorphism of $R$-algebras $\psi_j: B_{I_j} R \to S$ such that $\beta \circ \psi_j$ is the restriction of $\alpha$ to $1 \otimes \cdots \otimes 1 \otimes B_{I_j} R \otimes 1 \otimes \cdots \otimes 1$, and $\psi_1, \ldots, \psi_m$ together give a homomorphism $\varphi_0: B_{I_1} R \otimes_{R} \cdots \otimes_{R} B_{I_m} R \to S$ such that $\beta \circ \varphi_0 = \alpha$. Further we define $I_{n+1}:=I$. By definition, this is isomorphic to an ideal in $R$ and we can extend $\varphi_0$ to a map $\varphi: B_{I_1} R \otimes_{R} \cdots \otimes_{R} B_{I_{n+1}} R \to S$ which maps $1 \otimes \cdots \otimes 1 \otimes I_{n+1} \subseteq B_{I_1} R \otimes_{R} \cdots \otimes_{R} B_{I_{n+1}} R $ to $I \subseteq S_d$.

It remains to prove that $\varphi$ is an isomorphism. We first prove that it is surjective. It is sufficient to prove that the image of $\varphi$ contains every homogeneous element $t \in S$ and we show this by induction on $\deg t$. The case $\deg t =0$ is trivial, so we assume $\deg t>0$. We define $s:=\varphi_0(\alpha^{-1}(\beta(t))) \in \im \varphi_0 \subseteq \im \varphi \subseteq S$. Now $\beta \circ \varphi_0 = \alpha$ implies $\beta(s)=\beta(t)$, and thus $t-s \in \ker \beta =J$. Since $J$ is by definition generated by its degree-$d$-part, we can write $t-s=\sum_l r_l a_l$ with $r_l \in J_d$ and $a_l \in S$. Now by construction we obtain $r_l \in \im \varphi$ for every $l$, and since $\deg a_l = \deg t - d$, $a_l \in \im \varphi$ by induction. Finally the fact that $\im \varphi$ is a subalgebra of $S$ implies that $t=s+\sum_l r_l a_l \in \im \varphi$ and we thus proved that $\varphi$ is surjective. Since $\dim S \geq \dim T +1 = m+ \dim R +1=\dim (B_{I_1} R \otimes_{R} \ldots \otimes_{R} B_{I_{n+1}} R)$ by Proposition \ref{DimGenBlowup} and $B_{I_1} R \otimes_{R} \ldots \otimes_{R} B_{I_{n+1}} R$ is an integral domain by Lemmas \ref{GenBlowupRegular} and \ref{RegIntegral}, $\varphi$ is also injective. This finishes the proof.
\end{compactenum}
\vspace{-1em}
\end{proof}

Now we can prove the two theorems of this section:

\begin{proof}[Proof of Theorem \ref{StructureGradAlgDedekind}]
The implication $(ii) \implies (i)$ follows from Lemma \ref{GenBlowupRegular}. We prove the converse by induction on $\dim S \geq 1$. If $\dim S=1$, we use that $S/S_+ \cong R$ and thus $S_+$ is a prime ideal in $S$ which is not maximal; in fact, it is a minimal prime ideal because $\dim S=1$. Hence Lemma \ref{RegIntegral} implies $S_+=\{ 0 \}$ and thus $S=S_0$. Now we assume $\dim S >1$. Then $S \supsetneq R$; let $d$, $I$, and $T$ be as in Lemma \ref{StructureGradAlgLemma}. Then $T$ is again a graded ring with $T_0 \cong R$ and $\dim T < \dim S$. Since $T$ is again regular by Lemma \ref{StructureGradAlgLemma}a) we can apply induction and obtain $T \cong B_{I_1} R \otimes_{R} \cdots \otimes_{R} B_{I_m} R$ for certain ideals $I_i \subseteq R$. Now the theorem follows from Lemma \ref{StructureGradAlgLemma}c).
\end{proof}

\begin{proof}[Proof of Theorem \ref{StructureGradAlgFree}]
The implication $(ii) \implies (i)$ is clear. The strategy for the other implication is the same as in the proof of Theorem \ref{StructureGradAlgDedekind}. The case $\dim S=1$ works exactly as before and in the case $\dim S>1$ let again $d$, $I$, and $T$ be as in Lemma \ref{StructureGradAlgLemma}. Since $S_d$ is free, $I$ is principal, see \cite[Theorem 10.14]{jacba2}. By Lemma \ref{StructureGradAlgLemma}a) and b) we get that $T$ is regular and each $T_i$ is free. So we can apply induction and obtain that $T \cong R[x_1, \ldots, x_{n-1}]$. Since $I$ is principal and thus $B_I R \cong R[x]$, the theorem now follows from Lemma \ref{StructureGradAlgLemma}c).
\end{proof}

\section{Arithmetic invariants of pseudoreflection \\ groups} \label{SectionInvariants}

In this section we apply Theorems \ref{StructureGradAlgDedekind} and \ref{StructureGradAlgFree} to rings of arithmetic invariants of pseudoreflection groups. Let $R$ be a Dedekind domain which is not a field, $G$ a finite group, and consider an $R$-representation of $G$, that is, a group homomorphism $G \to Gl_n(R)$. We then define a $G$-action on the polynomial ring $R[x_1, \ldots, x_n]$ via $ (\sigma \cdot f)(a):=f(\sigma^{-1} \cdot a)$ for all $\sigma \in G, f \in R[x_1, \ldots, x_n], a \in R^n$. The object we want to study is the ring of invariants $R[x_1, \ldots, x_n]^G$. By a classical result of Noether (see Benson \cite[Theorem 1.3.1]{benson}) $R[x_1, \ldots, x_n]$ is a finitely generated $R[x_1, \ldots, x_n]^G$-module and hence $\dim R[x_1, \ldots, x_n]^G=\dim R[x_1, \ldots, x_n]= n +1$. Therefore, if $R[x_1, \ldots, x_n]^G$ is isomorphic to a polynomial ring $R[y_1, \ldots, y_m]$ or, more generally, to a tensor product of blowup algebras $B_{I_1, \ldots, I_m} R$, then $n=m$ by Proposition \ref{DimGenBlowup}. First we want to figure out the relation between $R[x_1, \ldots, x_n]^G$ and $(U^{-1}R)[x_1, \ldots, x_n]^G$ for a multiplicatively closed subset $U \subset R$. This is the content of the following lemma:

\begin{lemma} \label{InvarLocaliz}
\begin{compactenum}[a)]
Let $R,G,U$ be as above.
\item We have $(U^{-1} R)[x_1, \ldots, x_n]^G=U^{-1}(R[x_1, \ldots, x_n]^G)$.
\item If $R[x_1, \ldots, x_n]^G \cong B_{I_1, \ldots, I_n} R$, then 
\[ (U^{-1} R)[x_1, \ldots, x_n]^G \cong B_{U^{-1}I_1, \ldots, U^{-1} I_n} (U^{-1} R). \]
\item If $R[x_1, \ldots, x_n]^G \cong R[y_1, \ldots, y_n]$, then 
\[U^{-1} R[x_1, \ldots, x_n]^G \cong U^{-1}R[y_1, \ldots, y_n]. \]
\end{compactenum}
\end{lemma}

\begin{proof}
\begin{compactenum}[a)]
\item If $f \in (U^{-1} R)[x_1, \ldots, x_n]^G$, then $f$ is a multivariate polynomial with coefficients in $U^{-1}R$, so there exists a $u \in U$ such that $u \cdot f$ has coefficients in $R$, and hence $u \cdot f \in R[x_1, \ldots, x_n]$. But since $f$ is $G$-invariant, we have for all $\sigma \in G$: $\sigma \cdot(u \cdot f)= u \cdot (\sigma \cdot f)=u \cdot f$, so $u \cdot f \in R[x_1, \ldots, x_n]^G$ and thus $f \in U^{-1} R[x_1, \ldots, x_n]^G$. The reverse inclusion is clear.
\item For every ideal $I \subseteq R$ we have 
\[U^{-1} (B_I R)= U^{-1} (\bigoplus_{m \in \mathbb{N}_0} I^m)=\bigoplus_{m \in \mathbb{N}_0} U^{-1} I^m = \bigoplus_{m \in \mathbb{N}_0} (U^{-1} I)^m = B_{U^{-1} I} (U^{-1}R) . \]
Together with the assumption and part a), this implies
\begin{align*}
(U^{-1} R)[x_1, \ldots, x_n]^G &=U^{-1}(R[x_1, \ldots, x_n]^G) \cong U^{-1}(B_{I_1, \ldots, I_n} R) \\ &=U^{-1} (B_{I_1} R \otimes_R \ldots \otimes_R B_{I_n} R) \\ &\cong (U^{-1} B_{I_1} R) \otimes_{U^{-1} R} \ldots \otimes_{U^{-1} R} (U^{-1} B_{I_n} R) \\ &= (B_{U^{-1} I_1} U^{-1} R) \otimes_{U^{-1} R} \ldots \otimes_{U^{-1} R} (B_{U^{-1} I_n} U^{-1} R) \\ &= B_{U^{-1} I_1, \ldots, U^{-1} I_n} (U^{-1} R).
\end{align*}
\item is a special case of b).
\end{compactenum}
\vspace{-1em}
\end{proof}

The lemma implies in particular that if $R[x_1, \ldots, x_n]^G \cong B_{I_1, \ldots, I_n} R$, then for every prime ideal $\primId \subset R$, $R_\primId[x_1, \ldots, x_n]^G$ is a polynomial ring (note that $R_\primId$ is a discrete valuation ring and hence a principal ideal domain). The next theorem shows that the converse is also true:

\begin{theorem} \label{GenBlowupInv}
Let $R$ be a Dedekind domain, $G$ a finite group with a representation $G \to Gl_n(R)$. Assume that for every nonzero prime ideal $\primId$ in $R$, the ring of invariants $R_\primId[x_1, \ldots, x_n]^G$ is isomorphic to a polynomial ring over $R_\primId$. Then there exist nonzero ideals $I_1, \ldots, I_n$ in $R$ such that $R[x_1, \ldots, x_n]^G \cong B_{I_1, \ldots, I_n} R$. If for every $d \in \mathbb{N}$ the $R$-module $R[x_1, \ldots, x_n]^G_d$ is free, then $R[x_1, \ldots, x_n]^G \cong R[y_1, \ldots, y_n]$.
\end{theorem}

\begin{proof}
By Theorem \ref{StructureGradAlgDedekind} and Theorem \ref{StructureGradAlgFree}, it is sufficient to prove that the ring $R[x_1, \ldots, x_n]^G$ is regular. Hence, by Lemma \ref{RegHomMaxId}a) we have to prove that the localization $R[x_1, \ldots, x_n]^G_\maxId$ is regular for every homogeneous maximal ideal $\maxId \subset R[x_1, \ldots, x_n]^G$. By Lemma \ref{RegHomMaxId}b) we find that $\maxId=(\primId,R[x_1, \ldots, x_n]^G_+)$ for some nonzero prime ideal $\primId \subset R$. Now we find that $R[x_1, \ldots, x_n]^G_\maxId$ is a localization of $(R \backslash \primId)^{-1} R[x_1, \ldots, x_n]^G$. But $(R \backslash \primId)^{-1} R[x_1, \ldots, x_n]^G$ is by Lemma \ref{InvarLocaliz}a) isomorphic to $R_\primId[x_1, \ldots, x_n]^G$ and hence by assumption a polynomial ring over $R_\primId$ and thus regular. This implies that also $R[x_1, \ldots, x_n]^G_\maxId$ is regular.
\end{proof}

With these results our main question is reduced to the special case that $R$ is a discrete valuation ring. So in the following we assume that $R$ is a discrete valuation ring with maximal ideal $(\pi)$, $K \coloneqq \Quot(R)$, and $F \coloneqq R/(\pi)$. As above, let further $G$ be a finite group with an $R$-representation $G \to Gl_n(R)$. Then this induces representations $G \to Gl_n(K)$ and $G \to Gl_n(F)$. First of all, Lemma \ref{InvarLocaliz} implies that if $R[x_1, \ldots, x_n]^G$ is a polynomial ring, then $K[x_1, \ldots, x_n]^G$ is also a polynomial ring. The converse however is not true (see Example \ref{ExampleInvar}), so our goal is now to prove some conditions under which $R[x_1, \ldots, x_n]^G$ is a polynomial ring provided that $K[x_1, \ldots, x_n]^G$ is a polynomial ring. For this we will make use of the ring of invariants $F[x_1, \ldots, x_n]^G$ over the residue field $F=R/(\pi)$.

We first consider an easier question: suppose we have elements $f_1, \ldots, f_n \in R[x_1, \ldots, x_n]^G$ which are algebraically independent over $R$ with $K[x_1, \ldots, x_n]^G=K[f_1, \ldots, f_n]$. Under which conditions is $R[x_1, \ldots, x_n]^G=R[f_1, \ldots, f_n]$? This is answered by the following lemma.

\begin{lemma} \label{AInvModulo}
With the above notation, we have $R[x_1, \ldots, x_n]^G=R[f_1, \ldots, f_n]$ if and only if the classes of $f_1, \ldots, f_n$ in $F[x_1, \ldots, x_n]^G$ are algebraically independent.
\end{lemma}

\begin{proof}
First assume that there is an invariant $f \in R[x_1, \ldots, x_n]^G \backslash R[f_1, \ldots, f_n]$. Since $f \in R[x_1, \ldots, x_n]^G \subseteq K[x_1, \ldots, x_n]^G =K[f_1, \ldots, f_n]$, there is a polynomial $p \in K[y_1, \ldots, y_n]$ such that $f=p(f_1, \ldots, f_n)$. Then there is an $l \in \mathbb{N}$ such that $\pi^l \cdot f \in R[f_1, \ldots, f_n]$. We may assume that $l$ is minimal with this property, and since $f \notin R[f_1, \ldots, f_n]$, we have $l>0$. There is a polynomial $q \in R[y_1, \ldots, y_n]$ such that $\pi^l f=q(f_1, \ldots, f_n)$. By the algebraic independence of $f_1, \ldots, f_n$, $q=\pi^l p$. By the minimality of $l$, not all coefficients of $q$ are divisible by $\pi$, so the image $\overline{q}$ of $q$ in $F[y_1, \ldots, y_n]$ is not zero. On the other hand, since $l>0$, the class of $\pi^l f$ in $F[x_1, \ldots, x_n]$ is zero, so $\pi^l f=q(f_1, \ldots, f_n)$ implies a nontrivial relation between the classes of $f_1, \ldots, f_n$ over $F$, so these classes are algebraically dependent.

Conversely, if the classes of $f_1, \ldots, f_n$ are algebraically dependent, there exists a polynomial $0 \neq \overline{q} \in F[y_1, \ldots, y_n]$ such that $\overline{q}(\overline{f_1}, \ldots, \overline{f_n})=0$. We choose a polynomial $q \in R[y_1, \ldots, y_n]$ which is mapped to $\overline{q}$ under the canonical projection $R[y_1, \ldots, y_n] \to F[y_1, \ldots, y_n]$. Then $q(f_1, \ldots, f_n)$ must be divisible by $\pi$, so there is an $f \in R[x_1, \ldots, x_n]$ with $\pi f =q(f_1, \ldots, f_n)$. Since $f_1, \ldots, f_n$ are invariants, $f$ is also an invariant. Furthermore, $f_1, \ldots, f_n$ are algebraically independent over $K$, so there is only one polynomial $p \in K[y_1, \ldots, y_n]$ such that $f=p(f_1, \ldots, f_n)$, and by the above we have $p=\frac{1}{\pi} q$. But since $\overline{q} \neq 0$, we have $p \notin R[y_1, \ldots, y_n]$ and, by the uniqueness of $p$, this implies $f \notin R[f_1, \ldots, f_n]$. This finishes the proof.
\end{proof}

The next proposition gives a sufficient condition for $R[x_1, \ldots, x_n]^G$ to be a polynomial ring.

\begin{prop} \label{PolyRingSameDeg}
Let $R$ be a discrete valuation ring with maximal ideal $(\pi)$, $F:=R/(\pi)$ and $K:=\Quot(R)$. Let $G$ be a finite group with a representation $G \to Gl_n(R)$. If the rings of invariants $K[x_1, \ldots, x_n]^G$ and $F[x_1, \ldots, x_n]^G$ are polynomial rings whose generators have the same degrees, then $R[x_1, \ldots, x_n]^G$ is also a polynomial ring.
\end{prop}

The proof requires the following statement due to Kemper \cite[Proposition 16]{kemperrefl}, which we will need again later.

\begin{lemma} \label{LemmaKemper}
Let $G$ be a finite group and let $G \to Gl_n(K)$ be a representation of $G$ over a field $K$. If the representation is faithful, then for elements $f_1, \ldots, f_n \in K[x_1, \ldots, x_n]^G$ the following two properties are equivalent:
\begin{compactenum}[(i)]
\item The polynomials $f_1, \ldots, f_n$ are algebraically independent and we have $\deg f_1 \cdots \deg f_n=|G|$.
\item $K[x_1, \ldots, x_n]^G=K[f_1, \ldots, f_n]$.
\end{compactenum}
\end{lemma}

For the proof of Lemma \ref{LemmaKemper}, we refer to \cite{kemperrefl}.

\begin{proof}[Proof of Proposition \ref{PolyRingSameDeg}]
We first note that by the assumptions for every degree $d$ we have
\[ \dim_K K[x_1, \ldots, x_n]^G_d=\dim_F F[x_1, \ldots, x_n]^G_d. \]
Since $R$ is a principal ideal domain, $R[x_1, \ldots, x_n]^G_d$ is a free $R$-module; let $B=\{b_1, \ldots, b_k \}$ be a basis. Then $B$ is also a $K$-basis of $K[x_1, \ldots, x_n]^G_d$ and hence $k=|B|=\dim_K K[x_1, \ldots, x_n]^G_d$. Let $\lambda_1 , \ldots, \lambda_k \in R$ be such that the image of $\lambda_1 b_1 + \ldots + \lambda_k b_k$  under the canonical projection map $\varphi: R[x_1, \ldots, x_n]^G \to F[x_1, \ldots, x_n]^G$ is zero. Then $\frac{1}{\pi}(\lambda_1 b_1 + \ldots + \lambda_k b_k) \in R[x_1, \ldots, x_n]^G_d=\langle b_1, \ldots, b_k \rangle_R$, so by linear independence all $\lambda_i$ are divisible by $\pi$. Hence the image of $B$ under $\varphi$ is $F$-linearly independent and thus by the dimension formula from above an $F$-basis of $F[x_1, \ldots, x_n]^G_d$. In particular $F[x_1, \ldots, x_n]^G_d$ is contained in the image of $\varphi$. Since this is true for all $d$, we obtain that $\varphi$ is surjective. Now we write $F[x_1, \ldots, x_n]^G=F[f_1, \ldots, f_n]$. By the above there exist homogeneous polynomials $g_1, \ldots, g_n \in R[x_1, \ldots, x_n]^G$ such that for every $i$ we have $\deg g_i = \deg f_i$ and $\varphi(g_i)=f_i$. Then $g_1, \ldots, g_m$ are algebraically independent over $R$ and their degrees are by assumption the same as the degrees of generators for $K[x_1, \ldots, x_n]^G$. This implies together with Lemma \ref{LemmaKemper} that they are generators for $K[x_1, \ldots, x_n]^G$ themselves. Further their images under $\varphi$ are algebraically independent, and thus they are indeed generators for $R[x_1, \ldots, x_n]^G$ by Lemma \ref{AInvModulo}.
\end{proof}

An important special case in which the assumptions of Proposition \ref{PolyRingSameDeg} are satisfied is the following. We call a finite subgroup of $Gl_n(R)$ a pseudoreflection group if it is a pseudoreflection group in $Gl_n(K)$, i.e. if it is generated by pseudoreflections in $Gl_n(K)$,

\begin{cor} \label{InvDVRInvert}
Let $R$ be a discrete valuation ring with maximal ideal $(\pi)$, $F:=R/(\pi)$, and $K:=\Quot(R)$. Let $G \subseteq Gl_n(R)$ be a finite pseudoreflection group such that $|G|$ is invertible in $R$. Then $R[x_1, \ldots, x_n]^G$ is a polynomial ring.
\end{cor}

For the proof we need the following stronger version of Lemma \ref{LemmaKemper}, which I did not find in the literature.

\begin{lemma} \label{DegProduct}
Let $K$ be a field, $G$ a finite group, and $G \to Gl_n(K)$ a faithful $K$-representation of $G$ such that the ring of invariants $K[x_1, \ldots, x_n]^G$ is a polynomial ring. Assume that there exist algebraically independent homogeneous polynomials $f_1, \ldots, f_n \in K[x_1, \ldots, x_n]^G$ such that $\deg f_1 \cdots \deg f_n \leq |G|$, then $K[x_1, \ldots, x_n]^G=K[f_1, \ldots, f_n]$.
\end{lemma}

\begin{proof}
From Lemma \ref{LemmaKemper} we know that $K[x_1, \ldots, x_n]^G=K[f_1, \ldots, f_n]$ if and only if $\deg f_1 \cdots \deg f_n=|G|$, so we have to show that $\deg f_1 \cdots \deg f_n < |G|$ cannot be the case. We assumed that $K[x_1, \ldots, x_n]^G$ is a polynomial ring, so there exist homogeneous invariants $g_1, \ldots, g_n$ such that $K[x_1, \ldots, x_n]^G=K[g_1, \ldots, g_n]$. We may change the order of the $f_i$ and $g_i$ in such a way that for $i<j$ we have $\deg f_i \leq \deg f_j$ and $\deg g_i \leq \deg g_j$. Again by Lemma \ref{LemmaKemper} we have $\deg g_1 \cdots \deg g_n = |G|$, and hence if $\deg f_1 \cdots \deg f_n < |G|$ we have $d:=\deg f_i < \deg g_i$ for some index $i$. Then every element of $K[x_1, \ldots, x_n]^G_{\leq d}$ is contained in $K[g_1, \ldots, g_{i-1}]$ and thus the transcendence degree of the $K$-algebra $S$ generated by $K[x_1, \ldots, x_n]^G_{\leq d}$ is at most $i-1$. But this is a contradiction since $f_1, \ldots, f_i$ are algebraically independent elements of $S$.
\end{proof}

\begin{proof}[Proof of Corollary \ref{InvDVRInvert}]
Since $|G|$ is invertible in $R$, it is also invertible in $K$ and in $F$ and thus $K[x_1, \ldots, x_n]^G$ and $F[x_1, \ldots, x_n]^G$ are polynomial rings by the Chevalley-Shephard-Todd theorem. Let $f_1, \ldots, f_n$ be homogeneous generators of $F[x_1, \ldots, x_n]^G$. Then $\deg f_1 \cdots \deg f_n \leq |G|$ by Lemma \ref{LemmaKemper}. Let $g_i^{(0)}$ be any polynomial in $R[x_1, \ldots, x_n]$ which is mapped to $f_i$ by the projection map $\varphi : R[x_1, \ldots, x_n] \to F[x_1, \ldots, x_n]$; we define
\[ g_i:= \mathcal{R}(g_i^{(0)}) := \frac{1}{|G|} \sum_{\sigma \in G} \sigma \cdot g_i^{(0)} . \]
Here $\mathcal{R}$ is the Reynolds operator $R[x_1, \ldots, x_n] \to R[x_1, \ldots, x_n]^G$. Then $g_i$ is an invariant which is still mapped to $f_i$ under $\varphi$ because $\varphi(\sigma \cdot g_i^{(0)})=f_i$ for every $\sigma \in G$ since $f_i$ is an invariant. In particular, $g_i$ is homogeneous with $\deg g_i=\deg f_i$ and $g_1, \ldots, g_n$ are algebraically independent; Lemma \ref{DegProduct} implies that $K[x_1, \ldots, x_n]^G=K[g_1, \ldots, g_n]$.  Thus the assumptions of Proposition \ref{PolyRingSameDeg} are satisfied and we obtain that $R[x_1, \ldots, x_n]^G$ is a polynomial ring.
\end{proof}

By combining Corollary \ref{InvDVRInvert} with Proposition \ref{GenBlowupInv}, we immediately obtain the following theorem, which is the main result of this paper.

\begin{theorem} \label{InvPseudoGenBlowup}
Let $R$ be a Dedekind domain, and let $G \subset Gl_n(R)$ be a finite pseudoreflection group such that $|G|$ is invertible in $R$. Then there exist nonzero ideals $I_1, \ldots, I_n \subseteq R$ such that $R[x_1, \ldots, x_n]^G \cong B_{I_1, \ldots, I_n} R$.

 Furthermore, if $R[x_1, \ldots, x_n]^G_d$ is free for every $d$, then $R[x_1, \ldots, x_n]^G$ is isomorphic to a polynomial ring over $R$.
\end{theorem}

We write down an important special case of this theorem explicitly:

\begin{cor}
Let $R$ be a principal ideal domain, and let $G \subseteq Gl_n(R)$ be a finite pseudoreflection group such that $|G|$ is invertible $R$. Then $R[x_1, \ldots, x_n]^G$ is isomorphic to a polynomial ring over $R$.
\end{cor}

If we add an additional assumption on the representation, a sort of converse of Proposition \ref{PolyRingSameDeg} also holds.

\begin{prop} \label{PolyRingSameDegEquiv}
Let $R$ be a discrete valuation ring with maximal ideal $(\pi)$, $F:=R/(\pi)$, and $K:=\Quot(R)$. Let $G$ be a finite group together with a faithful $R$-representation $G \to Gl_n(R)$. Assume that $K[x_1, \ldots, x_n]^G$ is isomorphic to a polynomial ring, say $K[x_1, \ldots, x_n]^G=K[f_1, \ldots, f_n]$. Then the following two statements are equivalent:
\begin{compactenum}[(i)]
\item $F[x_1, \ldots, x_n]^G$ is isomorphic to a polynomial ring, say $F[x_1, \ldots, x_n]^G=F[g_1, \ldots, g_n]$, and for every $i \in \{ 1, \ldots, n\}$ we have $\deg f_i = \deg g_i$.
\item $R[x_1, \ldots, x_n]^G$ is a polynomial ring, and the group homomorphism $G \to Gl_n(F)$ given by the representation $G \to Gl_n(R)$ and the projection $R \to F$ is injective.
\end{compactenum}
\end{prop}

\begin{proof}
Let $\varphi$ denote the homomorphism $G \to Gl_n(F)$ mentionend in $(ii)$. We first prove the implication $(i) \implies (ii)$. The fact that $R[x_1, \ldots, x_n]^G$ is a polynomial ring is just Proposition \ref{PolyRingSameDeg}, and the injectivity of $\varphi$ follows since by Lemma \ref{LemmaKemper} we have $|G|=\deg f_1 \cdots \deg f_n$ and $|\im(\varphi)|=\deg g_1 \cdots \deg g_n$; the equality of these two products follows from $(i)$.

The proof of the implication $(ii) \implies (i)$ goes as follows. By assumption $R[x_1, \ldots, x_n]^G$ is isomorphic to a polynomial ring, say $R[x_1, \ldots, x_n]^G=R[h_1, \ldots, h_n]$. Then we have $\deg(h_1) \cdots \deg(h_n)=|G|$ and by Lemma \ref{AInvModulo}, the classes of $h_1, \ldots, h_n$ in $F[x_1, \ldots , x_n]$ are algebraically independent. By our assumption on $\varphi$, the representation $G \to Gl_n(F)$ is faithful, so we can now apply Lemma \ref{LemmaKemper} and thus get $F[x_1, \ldots, x_n]^G=F[\overline{h_1}, \ldots, \overline{h_n}]$. This proves the claim since after reordering the $h_i$ we have $\deg(f_i)=\deg(h_i)$ for all $i$.
\end{proof}

This together with Proposition \ref{GenBlowupInv} immediately implies the following result for rings of invariants over arbitrary Dedekind domains:

\begin{theorem} \label{AInvInjGrp}
Let $R$ be a Dedekind domain with $K:=\Quot(R)$, and let $G \subseteq Gl_n(R)$ be a finite pseudoreflection group such that $K[x_1, \ldots, x_n]^G$ is a polynomial ring $K[f_1, \ldots, f_n]$. Assume that for every nonzero prime ideal $\primId \subset R$ with $|G| \in \primId$ the homomorphism $G \to Gl_n(R/\primId)$ given by the inclusion $G \to Gl_n(R)$ and the projection $R \to R/\primId$ is injective. Then $R[x_1, \ldots, x_n]^G$ is isomorphic to $B_{I_1, \ldots, I_n} R$ for some nonzero ideals $I_1, \ldots, I_n \subseteq R$ if and only if for every nonzero prime ideal $\primId \subset R$ with $|G| \in \primId$ the ring of invariants $(R/\primId)[x_1, \ldots, x_n]^G$ is a polynomial ring $(R/\primId)[g_1, \ldots, g_n]$ with $\deg g_i=\deg f_i$ for each $i \in \{1 , \ldots, n \}$.
\end{theorem}

Again we write down the result for the special case that $R$ is a principal ideal domain explicitly.

\begin{cor}
Let $R$ be a principal ideal domain with $K:=Quot(R)$, and let $G \subseteq Gl_n(R)$ be a finite pseudoreflection group such that $K[x_1, \ldots, x_n]^G$ is a polynomial ring $K[f_1, \ldots, f_n]$. Assume that for every prime element $p \in R$ which divides $|G|$ the homomorphism $G \to Gl_n(R/(p))$ given by the inclusion $G \to Gl_n(R)$ and the projection $R \to R/(p)$ is injective. Then $R[x_1, \ldots, x_n]^G$ is isomorphic to a polynomial ring if and only if for every prime element $p \in R$ which divides $|G|$ the ring of invariants $(R/(p))[x_1, \ldots, x_n]^G$ is a polynomial ring $(R/(p))[g_1, \ldots, g_n]$ with $\deg g_i = \deg f_i$ for each $i \in \{ 1, \ldots, n\}$.
\end{cor}

The following example shows that in the situation of Proposition \ref{PolyRingSameDegEquiv}, it is possible that $F[x_1, \ldots, x_n]^G$ is a polynomial ring, but the degrees of the generators are not the same as for $K[x_1, \ldots, x_n]^G$ and therefore $R[x_1, \ldots, x_n]^G$ is not a polynomial ring.

\begin{example} \label{ExampleInvar}
We consider the discrete valuation ring $R=\mathbb{Z}_{(3)}=\{ \frac{a}{b} | a,b \in \mathbb{Z}, 3 \nmid b \}$ and the group $G=S_3$. It can easily be checked that we can define an $R$-representation of $G$ by mapping the transpositions $(1 , 2)$ and $(2 , 3)$ to $\begin{pmatrix} 1 & 3 \\ 0 & -1 \end{pmatrix}$ and $\begin{pmatrix} -2 & -3 \\ 1 & 2 \end{pmatrix}$, respectively. Then we have $K=\Quot(R)=\mathbb{Q}$ and $F=R/(3) \cong \mathbb{F}_3$. Since the two matrices given above define reflections, the ring of invariants $K[x,y]^G$ is a polynomial ring. Explicitly, $K[x,y]^G=K[f,g]$ with $f \coloneqq x^2+3xy+3y^2$ and $g \coloneqq 2x^3+9x^2y+9xy^2$. We can also compute the ring of invariants $F[x,y]^G$ explicitly. It is also a polynomial ring, which is generated by the two polynomials $f':=x$ and $g':=x^4y^2+x^2y^4+y^6$. Since we have $\deg(f') \cdot \deg(g')=|G|$, this implies with Lemma \ref{LemmaKemper} that the group homomorphism $G \to Gl_2(F)$ is injective. This could of course also have been checked directly. Now Proposition \ref{PolyRingSameDegEquiv} implies that $R[x,y]^G$ is not a polynomial ring. Indeed, suppose that $R[x,y]^G$ is a polynomial ring. Then the two generators of this $R$-algebra also generate $K[x,y]^G$ and thus are of degrees $2$ and $3$. This implies they are scalar multiples of $f$ and $g$. Since the coefficients of $f$ do not have a common divisor in $R$, and the same is true for the coefficients of $g$, we find that $R[x,y]^G=R[f,g]$. But this is not true since the invariant $g^2-4f^3$ is divisible by $27$, and because $f$ and $g$ are algebraically independent $h:=\frac{1}{27}(g^2-4f^3) \notin R[f,g]$. As we know from the proof of Proposition \ref{PolyRingSameDeg}, these results are also related to the projection map $\varphi: R[x,y]^G \to F[x,y]^G$. Indeed, we have $\varphi(f)=f'^2$, $\varphi(g)=f'^3$, and $\varphi(h)=g'$. So the image of $R[f,g]$ under $\varphi$ has transcendence degree one over $F$, and also does not contain $f'$, but only $f'^2$ and $f'^3$.
\end{example}

We did not find any example of a finite pseudoreflection group $G$ over a Dedekind domain $R$ with $|G| \in R^\times$ such that $R[x_1, \ldots, x_n]^G$ is not isomorphic to a polynomial ring, but only to a tensor product of blowup algebras, so we make the following conjecture.

\begin{conjecture} \label{ConjPseudoFree}
Let $R$ be a Dedekind domain and let $G \subseteq Gl_n(R)$ be a finite pseudoreflection group. Then $R[x_1, \ldots, x_n]^G$ is isomorphic to a polynomial ring over $R$.
\end{conjecture}

By Theorem \ref{StructureGradAlgFree} this is equivalent to the conjecture that $R[x_1, \ldots, x_n]^G_d$ is a free module for every $d \in \mathbb{N}$. We also state this conjecture more generally, not limited to pseudoreflection groups.

\begin{conjecture}
Let $R$ be a Dedekind domain and let $G$ be a finite group with an $R$-representation $G \to Gl_n(R)$. Then for every $d \in \mathbb{N}$, the $R$-module $R[x_1, \ldots, x_n]_d^G$ is free.
\end{conjecture}

\bibliographystyle{plain}
\bibliography{AInvPseudo}

\end{document}